\documentclass[journal]{IEEEtran}
\usepackage{cite}
\usepackage{amsmath,amssymb,amsfonts}
\usepackage{algorithmic}
\usepackage{graphicx}
\usepackage{textcomp}

\usepackage{amsmath,amssymb,amsthm}
\usepackage{xcolor}
\usepackage{xspace}
\usepackage{textcomp}
\usepackage{soul}

\usepackage{bm,fixmath} 

\usepackage{siunitx}
\usepackage{accents}

\usepackage{hyperref}

\usepackage{subcaption}

\DeclareMathOperator*{\argmin}{arg\,min}

\DeclareMathOperator{\diag}{diag}

\DeclareMathOperator{\dist}{dist}
\DeclareMathOperator{\proj}{proj}

\newtheorem{remark}{Remark}
\newtheorem{example}{Example}
\newtheorem{proposition}{Proposition}

\newtheorem{assumption}{Assumption}
\newtheorem{lemma}{Lemma}
\newtheorem{corollary}{Corollary}

\newcommand{\norm}[1]{\left\lVert#1\right\rVert}

\newcommand{\ubar}[1]{\underaccent{\bar}{#1}}

\def\N{{\mathbb{N}}}
\def\R{{\mathbb{R}}}

\def\av{{\mathbold{a}}}
\def\bv{{\mathbold{b}}}
\def\cv{{\mathbold{c}}}
\def\dv{{\mathbold{d}}}
\def\e{{\mathbold{e}}}

\def\uv{{\mathbold{u}}}

\def\w{{\mathbold{w}}}
\def\x{{\mathbold{x}}}
\def\y{{\mathbold{y}}}

\def\bphi{{\mathbold{\varphi}}}

\def\A{{\mathbold{A}}}
\def\B{{\mathbold{B}}}
\def\D{{\mathbold{D}}}
\def\E{{\mathbold{E}}}
\def\F{{\mathbold{F}}}
\def\G{{\mathbold{G}}}

\def\K{{\mathbold{K}}}
\def\I{{\mathbold{I}}}
\def\P{{\mathbold{P}}}
\def\Q{{\mathbold{Q}}}
\def\Rm{{\mathbold{R}}}
\def\U{{\mathbold{U}}}
\def\T{{\mathbold{T}}}
\def\V{{\mathbold{V}}}
\def\W{{\mathbold{W}}}
\def\X{{\mathbold{X}}}
\def\Y{{\mathbold{Y}}}
\def\Lm{{\mathbold{\Lambda}}}

\def\0{{\mathbf{0}}}
\def\1{{\mathbf{1}}}

\def\lmin{{\ubar{\lambda}}}
\def\lmax{{\bar{\lambda}}}

\def\Ts{{T_\mathrm{s}}}

\begin{document}
\title{Internal Model-Based Online Optimization}

\author{Nicola Bastianello, \IEEEmembership{Member, IEEE}, Ruggero Carli, \IEEEmembership{Member, IEEE}, and Sandro Zampieri, \IEEEmembership{Fellow, IEEE}
\thanks{Part of this work was supported by the Research Project PRIN 2017 Advanced Network Control of Future Smart Grids funded by the Italian Ministry of University and Research (2020-2023) http://vectors.dieti.unina.it and partially by MIUR (Italian Minister for Education) under the initiative ``Departments of Excellence" (Law 232/2016).}
\thanks{N. Bastianello is with the School of Electrical Engineering and Computer Science, KTH Royal Institute of Technology, Sweden (e-mail: nicolba@kth.se). The main part of the work was carried out when N. Bastianello was at the University of Padova, Italy.}
\thanks{R. Carli and S. Zampieri are with the Department of Information Engineering (DEI), University of Padova, Italy (e-mail: carlirug@dei.unipd.it, zampi@dei.unipd.it)}
}

\maketitle

\begin{abstract}
In this paper we propose a model-based approach to the design of online optimization algorithms, with the goal of improving the tracking of the solution trajectory (trajectories) w.r.t. state-of-the-art methods. We focus first on quadratic problems with a time-varying linear term, and use digital control tools (a robust internal model principle) to propose a novel online algorithm that can achieve zero tracking error by modeling the cost with a dynamical system. We prove the convergence of the algorithm for both strongly convex and convex problems. We further discuss the sensitivity of the proposed method to model uncertainties and quantify its performance. We discuss how the proposed algorithm can be applied to general (non-quadratic) problems using an approximate model of the cost, and analyze the convergence leveraging the small gain theorem. We present numerical results that showcase the superior performance of the proposed algorithms over previous methods for both quadratic and non-quadratic problems.
\end{abstract}


\section{Introduction}\label{sec:introduction}
In applications ranging from control \cite{liao-mcpherson_semismooth_2018,paternain_realtime_2019}, to signal processing \cite{hall_online_2015,fosson_centralized_2021,natali_online_2021}, to machine learning \cite{shalev_online_2011,dixit_online_2019,chang_distributed_2020}, recent technological advances have made the class of \emph{online optimization} problems of central importance. These problems are characterized by cost functions that vary over time, capturing the complexity of dynamic environments and changing optimization goals \cite{dallanese_optimization_2020,simonetto_timevarying_2020}. Online problems require the design of solvers that can be applied in real time, with the objective of tracking the time-varying solution(s) within some precision.
Formally, we are interested in solving the following
\begin{equation}\label{eq:generic-online-optimization}
	\x_k^* = \argmin_{\x \in \R^n} f_k(\x)
\end{equation}
where consecutive problems arrive $\Ts > 0$ seconds apart. A first approach to developing online algorithms is to repurpose methods from static optimization (say, gradient descent) by applying them to each problem in the sequence. The convergence properties of the resulting \emph{unstructured} \cite{simonetto_timevarying_2020} online algorithms have been studied for different classes of problems, \emph{e.g.} \cite{dixit_online_2019,shalev_online_2011,hall_online_2015,fosson_centralized_2021}. In general, it is possible to prove the convergence of these algorithms' output to a neighborhood of the dynamic solution trajectory (assuming, for simplicity, its uniqueness), and to bound its radius as a function of the problem properties \cite{dallanese_optimization_2020,simonetto_timevarying_2020}. 

However, unstructured algorithms are passive with respect to the dynamic nature of online problems, since they treat each problem in the sequence as self-standing. To remedy this fact, a broad part of the literature has been devoted to the design and analysis of \emph{structured} algorithms \cite{simonetto_timevarying_2020}, which aim to leverage knowledge of the problem in order to improve the performance by reducing the tracking error. Usually, this knowledge comes in the form of as a set of assumptions that bound the rate of change of the problem, say by bounding the difference between consecutive cost functions. A particular class of structured algorithms is that of \emph{prediction-correction} methods, in which we build an approximation of a future problem from past observations -- the prediction -- and use it to warm-start the solution of the next problem \cite{simonetto_class_2016,simonetto_prediction_2017,fazlyab_prediction_2018,li_online_2021}. Structured methods have been shown to achieve smaller tracking error than their unstructured counterparts, highlighting their benefits. In this paper we are interested in carrying this structured approach to online optimization one step further, \emph{by exploiting control-theoretical tools to design novel online algorithms which can track the optimal trajectory with greater precision than unstructured methods.}

Before we detail our proposed contribution, we review some works at the fruitful intersection of control theory and (online) optimization. An important class of online problems is that analyzed in \emph{feedback optimization}, in which the goal is to drive the output of a system to the optimal trajectory of an online optimization problem \cite{bernstein_online_2019,colombino_online_2020,hauswirth_timescale_2021}. This set-up, which includes model predictive control applications \cite{liao-mcpherson_semismooth_2018}, is characterized by the interconnection of an online algorithm and a system, and control theory has been leveraged to analyze the stability of this closed loop. 
While in feedback optimization we employ optimization techniques to solve a control problem, control theory has also proved useful in the design and analysis of optimization algorithms, see \emph{e.g.} \cite{lessard_analysis_2016,sundararajan_canonical_2019,zhang_unified_2021,scherer_convex_2021,franca_gradient_2021}.
Finally, we mention \cite{shahrampour_online_2017,shahrampour_distributed_2018} in which online algorithms are developed under the assumption that the optimal trajectory evolves according to a linear dynamical system.

Before presenting the proposed contribution, we discuss two motivating examples.

\begin{example}[Online learning]\label{ex:online-learning}
Consider an online learning problem, in which we need to train a learner based on time-varying data \cite{shalev_online_2011}. We are then interested in algorithms that can exploit a (possibly rudimentary) model of the cost variation to improve accuracy of the learned model in the long run.
\end{example}

\begin{example}[Signal processing over sensors network]
Consider a signal processing problem in which the data are provided by measurements taken by a networks of sensors. In this context, one can envision two sources of time-variability, one stemming from the dynamic nature of the process being monitored, and the other by the possible intermittent failure of the sensors to deliver a new measurement.
\end{example}

As discussed so far, the focus of this paper is designing structured online algorithms that rely on a model of the dynamic cost, as derived \emph{e.g.} from first principles or historical data, and apply control theoretical tools in order to achieve better performance in the long run. Assuming that the cost is smooth, our approach is to recast the online optimization problem as the goal of driving the gradient of the time-varying cost (the plant) to zero. In turn this implies perfect tracking of the optimal trajectory, an achievement which unstructured methods cannot reach in general. We first explore this idea as applied to quadratic problems with a time-varying linear term, which arise for example in signal processing \cite{dixit_online_2019} and machine learning \cite{shalev_online_2011}. In the context of quadratic problems we leverage tools from digital control to develop a novel online algorithm which, making use of a \emph{model} of the cost as a dynamical system, can achieve zero tracking error. In particular, we employ the internal model principle in combination with robust control theory to design the algorithm. We prove the convergence of the proposed method both for strongly convex (to the unique optimal trajectory) and convex problems (to the time-varying set of solutions). The proposed algorithm however relies on the knowledge of the cost model, which in practice may be inaccurate. Therefore, we analyze its sensitivity to inexact models, by characterizing a bound for the tracking error as a function of the model uncertainty. After laying the groundwork by focusing on a specific class of quadratic problems, we then discuss the application of the proposed algorithms to general quadratic (with a time-varying Hessian) and non-quadratic problems as well. We show how the proposed methods can be applied to any online (smooth) optimization problem with the use of an approximate cost model, and analyze their convergence to a neighborhood of the optimal trajectory by leveraging the small gain theorem. Finally, we showcase the better performance of the algorithms as compared to state-of-the-art alternatives, both for quadratic and non-quadratic problems. The contribution is summarized in the following points:

\begin{enumerate}	
	\item We propose a \emph{control-based online algorithm} that leverages a model of the cost to improve the performance over unstructured methods. In particular, we show that for quadratic problems with a time-varying linear term the algorithm tracks with zero error the optimal trajectory (for strongly convex problems) or the optimal set (for convex problems).
	
	\item We discuss the sensitivity of the proposed method to model uncertainties, and provide an upper bound to the tracking error achieved when an inexact model is used.
	
	\item We show how the proposed structured algorithm can be applied to general (non-quadratic) online problems by choosing a suitable approximate model. Interpreting the inexact model as a source of disturbance, we then provide a convergence analysis of the resulting algorithms.
	
	\item We present numerical results that showcase the superior performance of the proposed algorithms with respect to state-of-the-art methods, both for quadratic and non-quadratic problems.
\end{enumerate}

\smallskip

\paragraph*{Outline}
In section~\ref{sec:control-based-algorithm} we propose and discuss the control-based online algorithm. Section~\ref{sec:convergence} analyzes the convergence for quadratic and non-quadratic problems. Section~\ref{sec:simulations} presents some numerical results and section~\ref{sec:conclusions} concludes the paper.

\paragraph*{Notation}
We denote by $\N$, $\R$ the natural, real numbers, respectively, and by $\R[z]$ the space of polynomials in $z$ with real coefficients.
Vectors and matrices are denoted by bold letters, \emph{e.g.} $\x \in \R^n$ and $\A \in \R^{n \times m}$. $\I$ denotes the identity, $\0$ and $\1$ denote the vectors of all $0$s and $1$s. 
The Euclidean norm and inner product are $\norm{\cdot}$ and $\langle \cdot, \cdot \rangle$. $\dist(\x, \X) = \min_{\y \in \X} \norm{\x - \y}$ denotes the distance of $\x$ from a set $\X \subset \R^n$.
Given a transfer matrix $\B(z) \in \R^{n \times n}$ we write $\norm{\B(z)}_\infty = \max_{\theta \in [0, 2\pi]} \sigma_{\max}(\B(e^{j \theta}))$ where $\sigma_{\max}(\cdot)$ denotes the maximum singular value.
A positive semi-definite (definite) matrix is denoted as $\A \succeq 0$ ($\A \succ 0$). $\otimes$ denotes the Kronecker product.
$\diag$ denotes a (block) diagonal matrix built from the arguments.
Let $f : \R^n \to \R$, we denote by $\nabla f$ its gradient. $f$ is $\lmin$-strongly convex, $\lmin > 0$, iff $f - (\lmin/2) \norm{\cdot}^2$ is convex, and $\lmax$-smooth iff $\nabla f$ is $\lmax$-Lipschitz continuous.

\section{Control-based Online Optimization}\label{sec:control-based-algorithm}
We focus first on the class of following class of online quadratic problems
\begin{equation}\label{eq:online-quadratic}
	f_k(\x) := \frac{1}{2} \x^\top \A \x +  \x^\top b_k
\end{equation}
in which we assume that the symmetric matrix $\A$ is fixed and that only $\bv_k$ is time-varying. Restricting our attention to these problems will allow us to design the algorithm proposed in this section. In section~section~\ref{sec:convergence}, its convergence will be first proved in the context described in~\eqref{eq:online-quadratic}, and then extended to more general (not necessarily quadratic) problems.

In the following we assume that Assumption~\ref{as:generic-online-problem}~and~\ref{as:model-b} hold.

\medskip 
\begin{assumption}[Strongly convex and smooth]\label{as:generic-online-problem}
The symmetric matrix $\A$ is such that $\lmin \I \preceq \A = \A^\top \preceq \lmax \I$, with $0 < \lmin \leq \lmax < \infty$.
This is equivalent to imposing that the cost functions $\{ f_k \}_{k \in \N}$ are $\lmin$-strongly convex and $\lmax$-smooth for any time $k \in \N$.
\end{assumption}

\begin{assumption}[Model of $\bv_k$]\label{as:model-b}
The coefficients of the linear term $\{ \bv_k \}_{k \in \N}$ have a rational $\mathcal{Z}$-trasform, namely
\begin{equation}\label{eq:B-z-transform}
	\mathcal{Z}[\bv_k] = \B(z) = \frac{\B_\mathrm{N}(z)}{B_\mathrm{D}(z)}, \quad B_\mathrm{D}(z) = z^m + \sum_{i = 0}^{m-1} b_i z^i
\end{equation}
for some $\B_\mathrm{N}(z) \in \R^n[z]$, where the poles of $B_\mathrm{D}(z)$ are all marginally or asymptotically stable.
\end{assumption}
\medskip 

Assumption~\ref{as:generic-online-problem} is widely used in online optimization \cite{dallanese_optimization_2020}, since the strong convexity of the cost implies that each problem in~\eqref{eq:generic-online-optimization} has a unique minimizer, and we can define the optimal trajectory $\{ \x_k^* \}_{k \in \N}$. 
Assumption~\ref{as:model-b} instead provides the model of the cost function used for the design of the proposed algorithm. 
We rule out unstable modes in $\bv_k$, since they cause unbounded growth of $\norm{\x_k^* - \x_{k-1}^*}$, which is the distance between subsequent minimizers.  Observe that the assumption that $\norm{\x_k^* - \x_{k-1}^*}$ is bounded for all $k \in \N$ is standard in online optimization \cite{dallanese_optimization_2020}. Nevertheless the proposed model is general enough to capture any kind of vector-valued signal $\bv_k$ characterized by a finite number of frequencies. We remark moreover that \emph{only knowledge of the denominator} $B_\mathrm{D}(z)$ is required, whereas the numerator does not play any role in the algorithm.

\begin{remark}[Gradient oracle]
In the following we assume that only an \emph{oracle of the gradient} can be accessed by the online algorithm, alongside the bounds on $\A$'s eigenvalues $\lmin$, $\lmax$. Since the algorithm does not know $\A$, then the closed form solution of the optimization when~\eqref{eq:online-quadratic} holds, that is $\x_k^* = - \A^{-1} \bv_k$, cannot be computed. 
This assumption is in accordance with frameworks adopted by all gradient methods that our technique aims to improve.
\end{remark}

\begin{remark}[State of the art]\label{rem:online-gradient}
A widely used algorithm that can be applied to these time-varying problems is the \textit{online gradient} method characterized by $\x_{k+1} = \x_k - \alpha \nabla f_k(\x_k)$. Under Assumption~\ref{as:generic-online-problem} and provided that $\norm{\x_k^* - \x_{k-1}^*}$ is bounded for any $k$, then it can be proved that the online gradient satisfies $\limsup_{k \to \infty} \norm{\x_k - \x_k^*} \leq R < \infty$ when $\alpha \in (0, 2/\lmax)$ \cite{dallanese_optimization_2020}. The proposed control-based approach will allow us to design an algorithm that can instead guarantee $\limsup_{k \to \infty} \norm{\x_k - \x_k^*} = 0$ for~\eqref{eq:online-quadratic}.
\end{remark}

\subsection{Control scheme}\label{subsec:control-scheme}
In order to drive the gradient signal $\e_k := \nabla f_k(\x_k)$ to zero asymptotically, we set-up the control scheme depicted in Figure~\ref{fig:block-diagram}. In the block diagram we see that the gradient $\e_k$ is connected in feedback with a controller $C(z)$ -- to be designed -- which produces the output $\x_k$ that is the estimate of the minimizer $\x_k^*$ provided by the algorithm.

\begin{figure}[!ht]
\centering
\includegraphics[scale=0.8]{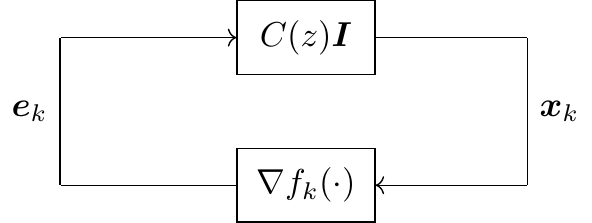}
\caption{The control scheme designed to solve~\eqref{eq:generic-online-optimization}. Recall that in case \eqref{eq:online-quadratic} holds, then $\nabla f_k(\x) = \A \x + \bv_k$.} 
\label{fig:block-diagram}
\end{figure}

Our goal then is to design $C(z)$ so that $\e_k$ converges to zero. We can do this by leveraging the \emph{internal model principle} \cite{fadali_digital_2019}, namely by choosing
\begin{equation}\label{eq:controller-structure}
	C(z) = \frac{C_\mathrm{N}(z)}{B_\mathrm{D}(z)}, \quad \text{with} \quad C_\mathrm{N}(z) = \sum_{i = 0}^{m-1} c_i z^i,
\end{equation}
where we have to find a polynomial $C_\mathrm{N}(z)$ of degree less than $m$ (to guarantee that $C(z)$ is \emph{strictly proper}) such that  the closed loop system is  asymptotically stable.

By Figure~\ref{fig:block-diagram} we argue that $\X(z) = C(z) \E(z)$ and $\E(z) = \A \X(z) + \B(z)$, and combining them with~\eqref{eq:controller-structure} we get
\begin{equation}\label{eq:z-transform-e}
	\E(z) = \left( B_\mathrm{D}(z) \I - C_\mathrm{N}(z) \A \right)^{-1} \B_\mathrm{N}(z). 
\end{equation}
Therefore the poles of $\E(z)$, when using the controller structure~\eqref{eq:controller-structure}, come from the term $\left( B_\mathrm{D}(z) \I - C_\mathrm{N}(z) \A \right)^{-1}$. Hence, to achieve convergence to zero, we need to choose $C_\mathrm{N}$ such that these poles are stable.
In the next section we will show how the design of $C_\mathrm{N}(z)$ can be translated into a robust control problem.

\subsection{Designing the robust stabilizing controller}\label{subsec:stabilizing-controller}

Observe first that the problem can be simplified. Consider the eigendecomposition of $\A = \V \Lm \V^\top$, with $\V^\top \V = \I$ and $\Lm = \diag\{ \lambda_i \}_{i = 1}^n$, $\lambda_i \in [\lmin, \lmax]$. Using it in~\eqref{eq:z-transform-e} yields
\begin{equation}\label{eq:z-transform-e-eigendecomposition}
	\E(z) = \V \left( B_\mathrm{D}(z) \I - C_\mathrm{N}(z) \Lm \right)^{-1} \V^\top \B_\mathrm{N}(z).
\end{equation}
We remark that, even though the eigendecomposition of $\A$ is used in this theoretical discussion, \emph{we do not need to know it to implement algorithm~\eqref{eq:control-based-algorithm}}.
By~\eqref{eq:z-transform-e-eigendecomposition} we see that the poles of $\E(z)$ in the scheme of Figure~\ref{fig:block-diagram} only come from the inverse of the diagonal matrix $B_\mathrm{D}(z) \I - C_\mathrm{N}(z) \Lm$. Namely, the poles of $\E(z)$ are the roots of the $n$ polynomials $\left\{ B_\mathrm{D}(z) - C_\mathrm{N}(z) \lambda_i \right\}_{i = 1}^n$,
$\lambda_i \in [\lmin, \lmax]$. Therefore, in order to ensure asymptotic stability, it is sufficient to choose a $C_\mathrm{N}(z)$ that stabilizes the polynomials
$$
	p(z;\lambda) := B_\mathrm{D}(z) - C_\mathrm{N}(z) \lambda = z^m + \sum_{i = 0}^{m-1} (b_i - \lambda c_i) z^i
$$
for all $\lambda \in [\lmin, \lmax]$. This ensures that, as long as the eigenvalues of $\A$ lie in the range $[\lmin, \lmax]$, the controller is stabilizing. This can be seen as a robust control problem, which we solve by solving an LMI as shown below.

Indeed, observe that the stability of $p(z;\lambda)$ is equivalent to the stability of  the associated companion matrix
$
	\F_\mathrm{c}(\lambda) := \F +\lambda \G \K,
$
with
\begin{align*}
	\F &= \begin{bmatrix}
		0 & 1 & 0 & \cdots \\
		& & \ddots & \\
		0 & \cdots & 0 & 1 \\
		-b_0 & \cdots & \cdots & -b_{m-1}
	\end{bmatrix},\quad
	\G= \begin{bmatrix}
		0\\
		\vdots\\
		0\\
		1
	\end{bmatrix}\\
\K&=\begin{bmatrix}c_0 & c_1 & \cdots & \cdots & c_{m-1}\end{bmatrix}
\end{align*}
Then the goal is to find $c_0,c_1,\ldots,c_{m-1}$  so that $\F_\mathrm{c}(\lambda)$  
is stable for all $\lambda \in [\lmin, \lmax]$. We can write $\lambda$ as the convex combination of the two extreme values by
$$
	\lambda = \alpha(\lambda)\lmin + (1 - \alpha(\lambda)) \lmax, \quad \alpha(\lambda) = \frac{\lmax - \lambda}{\lmax - \lmin}
$$
which allows us to use the following result \cite[Theorem~3]{de_oliveira_new_1999}.

\begin{lemma}[Stabilizing controller design]\label{lem:stabilizing-controller}
The matrix $\F_\mathrm{c}(\lambda)$ is asymptotically stable for all $\lambda \in [\lmin, \lmax]$ if there exist symmetric matrices $\underline\P, \overline\P\succ 0$, $\Q \in \R^{m \times m}$, and $\Rm \in \R^{1 \times m}$ such that the following two linear matrix inequalities are verified
$$
	\begin{bmatrix}
		\underline\P & \F \Q + \lmin\G \Rm \\
		\Q^\top \F^\top + \Rm^\top \lmin\G & \Q + \Q^\top -\underline\P
	\end{bmatrix} \succ 0, 
$$
$$
		\begin{bmatrix}
		\overline\P & \F \Q + \lmax\G \Rm \\
		\Q^\top \F^\top + \Rm^\top \lmax\G & \Q + \Q^\top -\overline\P
	\end{bmatrix} \succ 0.
$$
A stabilizing controller is then
$
	\K=\Rm \Q^{-1}.
$
\end{lemma}

\begin{remark}[Computing the controller]\label{rem:lmi-solution}
We remark that the LMIs' dimension depends on the degree $m$ of $B_\mathrm{D}(z)$, \emph{and not on the size of the problem $n$}.
We further remark that, depending on the model of $\bv_k$ and the bounds $\lmin$ and $\lmax$, there are cases in which the controller in Lemma~\ref{lem:stabilizing-controller} may not exist, that is, in which the LMIs are not feasible. While running the numerical experiments, we observed that, in some cases, the LMIs tend to be infeasible when the roots of $B_\mathrm{D}(z)$ are close to one and the condition number of $\A$ ($\lmax / \lmin$) is large.
Finally, while in the following we assume that a stabilizing controller does indeed exist and can be computed, if this is not the case we can fall back to the online gradient descent (cf. Remark~\ref{rem:online-gradient}).
\end{remark}

\subsection{Online algorithm}\label{subsec:online-algorithm}
We are finally ready to propose the online algorithm derived from the scheme of Figure~\ref{fig:block-diagram}.
Since $\X(z) = C(z) \E(z)$, and with the choice~\eqref{eq:controller-structure}, it is straightforward to see that the system defines the following online algorithm:
\begin{subequations}\label{eq:control-based-algorithm}
\begin{align}
	\w_{k+1} &= \left( \F \otimes \I \right) \w_k + (\G \otimes\I) \nabla f_k(\x_k) \label{eq:control-based-algorithm-w} \\
	\x_{k+1} &= \left( \K \otimes \I \right) \w_{k+1}
\end{align}
\end{subequations}
where, as defined above $\nabla f_k(\x_k) = \e_k = \A \x_k + \bv_k$, and $\w \in \R^{n m}$ is a set of $m$ auxiliary vectors, representing the state of the closed loop canonical control form realization.
Notice that the updates of algorithm~\eqref{eq:control-based-algorithm} are fully defined with knowledge of \textit{(i)} the internal model $B_\mathrm{D}(z)$, and \textit{(ii)} the bounds $\lmin$, $\lmax$, which, as discussed above, is a common assumption in (online) optimization. Indeed, of the model $\mathcal{Z}[\bv_k]$ we need only to know the poles, while no information on the numerator $\B_\mathrm{N}(z)$ is required.

\subsubsection*{Non-quadratic problems}
We remark that algorithm~\eqref{eq:control-based-algorithm} can be \textit{applied in a straightforward manner also to the more general (not necessarily quadratic) problem~\eqref{eq:generic-online-optimization}}, since its updates receive as input the gradient $\nabla f_k(\x_k)$. 
But, while for more general costs we can still access the bounds $\lmin$, $\lmax$\footnote{Which represent the strong convexity and smoothness moduli of $f_k$.}, we no longer have access to the internal model $B_\mathrm{D}(z)$. Instead, \textit{which internal model is used becomes a design parameter of the algorithm}. Tuning the internal model can be done by leveraging some information on the variability of the cost function (as discussed in Example~\ref{ex:periodic}), or it may be estimated from historical data.
Regardless of how the internal model is derived, when deriving it we need to keep in mind that smaller models (having lower degree $m$) are generally better. Indeed, larger models may lead to infeasibility of the controller LMIs (cf. Remark~\ref{rem:lmi-solution}), or much longer transients (cf. sections~\ref{subsec:numerical-tv-quadratic}~and~\ref{subsec:numerical-non-quadratic}).

\begin{example}[Periodic $f_k$]\label{ex:periodic}
Consider a problem~\eqref{eq:generic-online-optimization} in which the cost function is periodic with period $P$. In this case a reasonable choice of internal model is the transfer function
\begin{equation}\label{eq:approximate-internal-model}
	B_\mathrm{D}(z) = (z - 1) \prod_{\ell = 1}^L (z^2 - 2\cos(\ell \theta) z + 1)
\end{equation}
whose poles are multiples of the frequency $\theta = (2\pi / P) \Ts$\footnote{Recall that $\Ts$ is the ``sampling time'', \textit{i.e.} the time that elapses between the arrival of two consecutive cost functions.}.
\end{example}

\section{Convergence Analysis}\label{sec:convergence}
As discussed above, the proposed algorithm~\eqref{eq:control-based-algorithm} can be applied both to the quadratic problem, for which ~\eqref{eq:online-quadratic} holds, as well as to the more general problem~\eqref{eq:generic-online-optimization}.
The following sections then analyze the convergence of the algorithm in both scenarios, providing bounds to the tracking error $\norm{\x_k -\x_k^*}$ and discussing their implications.

\subsection{Convergence for quadratic problem~\eqref{eq:online-quadratic}}\label{subsec:convergence-A-ti}

\begin{proposition}[Convergence for~\eqref{eq:online-quadratic}]\label{pr:convergence-control-algorithm}
Assume that $f_k$ is quadratic as in ~\eqref{eq:online-quadratic} and that Assumptions~\ref{as:generic-online-problem}~and~\ref{as:model-b} hold. 
Assume that the controller $C(z) = \frac{C_\mathrm{N}(z)}{C_\mathrm{D}(z)}$ is such that
\begin{enumerate}
	\item[(c1)] \textit{internal model}: $C_\mathrm{D}(z)$ includes all poles of $B_\mathrm{D}(z)$,
	
	\item[(c2)] \textit{stability}: $B_\mathrm{D}(z) - C_\mathrm{N}(z) \lambda$ is stable for all $\lambda \in [\lmin, \lmax]$.
\end{enumerate}
Then the output $\{ \x_k \}_{k \in \N}$ of the online algorithm~\eqref{eq:control-based-algorithm} is such that
$$
	\limsup_{k \to \infty} \norm{\x_k - \x_k^*} = 0.
$$
\end{proposition}
\begin{proof}
By choosing a controller that stabilizes the matrix $\F_{\mathrm{c}}(\lambda)$, $\lambda \in [\lmin, \lmax]$, the poles of $\E(z)$ (cf.~\eqref{eq:z-transform-e-eigendecomposition}) are asymptotically stable, and the gradient $\e_k$ converges to zero, implying the thesis.
\end{proof}

\subsubsection*{Convex problems}
So far we have dealt, according to Assumption~\ref{as:generic-online-problem}, with strongly convex problems. In the following we show that the results derived in this section can be leveraged to prove convergence of~\eqref{eq:control-based-algorithm} for \emph{convex problems}. In particular, we are able to show that the output of~\eqref{eq:control-based-algorithm} converges linearly to the sub-space of solutions.

\begin{proposition}[Convergence for convex~\eqref{eq:online-quadratic}]
In the set-up of Proposition~\ref{pr:convergence-control-algorithm} assume that $\A = \A^\top \succeq 0$ with $\operatorname{rank}(\A) = r < n$, and assume that the non-zero eigenvalues lie in $[\lmin, \lmax]$. Assume that $\bv_k \in \operatorname{range}(\A)$ for all $k$ and let $\X_k^*$ be the (affine) set of solutions to~\eqref{eq:online-quadratic}. Assume moreover that the controller $C(z) = \frac{C_\mathrm{N}(z)}{C_\mathrm{D}(z)}$ verifies (c1) and (c2) of Proposition~\ref{pr:convergence-control-algorithm}.
Then the output $\{ \x_k \}_{k \in \N}$ of the online algorithm~\eqref{eq:control-based-algorithm} is such that
$$
	\limsup_{k \to \infty} \ \dist(\x_k, \X_k^*) = 0.
$$
\end{proposition}
\begin{proof}
Consider the eigendecomposition of the (now positive semi-definite) matrix $\A$, $\A = \V \Lm \V^\top$, where we let $\V = \begin{bmatrix} \U & \W \end{bmatrix}$, $\Lm = \diag\left\{ \lambda_1, \ldots, \lambda_r, 0, \ldots, 0 \right\}$, $\U \in \R^{n \times r}$ having as columns the eigenvectors of the non-zero eigenvalues. With this notation, we know that the affine set of solution is defined by
$
	\X_k^* = \left\{ - \A^+ \bv_k + (\I - \A^+ \A) \w \ | \ \w \in \R^n \right\}
$
where $\A^+$ is the pseudo-inverse
\begin{equation}\label{eq:pseudo-inverse-A}
	\A^+ = \V \diag\left\{ \lambda_1^{-1}, \ldots, \lambda_r^{-1}, 0, \ldots, 0 \right\} \V^\top.
\end{equation}

The projection of $\x_k$ onto $\X_k^*$ is defined as $\proj_{\X_k^*}(\x_k) := (\I - \A^+ \A) \x_k - \A^+ \bv_k$ \cite[section~6.2.2]{parikh_proximal_2014}, and our goal is to prove that, asymptotically, $\x_k - \proj_{\X_k^*}(\x_k) \to 0$. Indeed, since $\proj_{\X_k^*}(\x_k)$ is the element of $\X_k^*$ closest (in Euclidean norm) to $\x_k$, this implies that the distance of $\x_k$ from $\X_k^*$ converges to zero asymptotically.

By the definitions above, it is straightforward to see that
$$
	\x_k - \proj_{\X_k^*}(\x_k) = \A^+ \e_k
$$
and if we prove that $\A^+ \e_k \to 0$ as $k \to \infty$ then the thesis is proved. By~\eqref{eq:z-transform-e-eigendecomposition} we know that
$$
	\E(z) = \V \left( B_\mathrm{D}(z) \I - C_\mathrm{N}(z) \Lm  \right)^{-1} \V^\top \B_\mathrm{N}(z)
$$
where in the convex case the matrix $B_\mathrm{D}(z) \I - C_\mathrm{N}(z) \Lm$ has $n - r$ eigenvalues equal to $B_\mathrm{D}(z)$, and the remaining $n-r$ are $\left\{ B_\mathrm{D}(z) - C_\mathrm{N}(z) \lambda_i \right\}_{i = 1}^r$, $\lambda_i \in [\lmin, \lmax]$. Moreover, using~\eqref{eq:pseudo-inverse-A} we can see that
$$
	\A^+ \E(z) = \U \Lm_{>0}^{-1} \left( B_\mathrm{D}(z) \I - C_\mathrm{N} \Lm_{>0}^{-1} \right)^{-1} \U^\top \B_\mathrm{N}(z)
$$
where $\Lm_{>0}^{-1} = \diag\left\{ \lambda_1, \ldots, \lambda_r \right\}$.
Finally, if we choose a stabilizing controller then the poles of $\A^+ \E(z)$ are stable, and the thesis is proved.
\end{proof}

\subsubsection*{Piece-wise model variation}
An interesting observation is that the proposed algorithm can also be applied when $\bv_k$ obeys the model only in a piece-wise fashion.
Precisely, in the current set-up one can see $\bv_k$ as the output of the transfer matrix $\diag\{ \B(z) \}$ driven by the impulsive input $\uv_k = \av \delta(k)$, $\av \in \R^n$. More generally we can also drive the same transfer matrix with input $\uv_k = \sum_{i \in \N} \av_i \delta(k - k_i)$, $\av_i \in \R^n$, which results in a different model $\B_\mathrm{N}^i(z) / B_\mathrm{D}(z)$ when each new impulse acts.
For example, taking $B(z)=(z-1)^2$ we can model in this fashion a signal $\bv_k$ characterized by a sequence of ramp segments, each with a different slope.
When applied to this class of problems, the tracking error of the proposed algorithm does not converge asymptotically to zero. Rather, at every change in the model (at the times $k_i$) the algorithm undergoes a new transient, and then establishes convergence towards zero -- until a new change in the problem occurs.

\subsection{Convergence for the general problem~\eqref{eq:generic-online-optimization}}\label{subsec:convergence-generic}
So far, we have leveraged the particular class of quadratic problems~\eqref{eq:online-quadratic} to inspire the design of the proposed online algorithm, and proved in the previous section that for these problems the algorithm achieves perfect tracking of $\x_k^*$. Now we turn our attention to the more general problem~\eqref{eq:generic-online-optimization}
where we assume the costs $f_k$ satisfy the following assumption:

\begin{assumption}[General problem]\label{as:extensions}
The cost functions $\{ f_k \}_{k \in \N}$ can be decomposed as
\begin{equation}\label{eq:general-cost}
	f_k(\x) = \frac{1}{2} \x^\top \A \x +  \bv_k^\top \x  + \varphi_k(\x)
\end{equation}
where:
\begin{enumerate}
	\item[(i)] $\A $ satisfies Assumption~\ref{as:generic-online-problem};
	
	\item[(ii)] $\{ \bv_k \}_{k \in \N}$ satisfies Assumption~\ref{as:model-b} and additionally $\norm{\bv_k} \leq \beta$ for all $k \in \N$;
	
	\item[(iii)] $\varphi_k(\x) = \varphi_k'(\x) + \varphi_k''(\x)$ with
	$$
		\norm{\nabla \varphi_k'(\x)} \leq \delta \quad \text{and} \quad \norm{\nabla \varphi_k''(\x)} \leq \gamma \norm{\x}, \quad \forall \x \in \R^n.
	$$
\end{enumerate}
\end{assumption}

\vspace{0.35cm}

Before proving the main convergence result, let us comment on the choice of the cost~\eqref{eq:general-cost}. We can write the cost at hand as
$
	f_k(\x) = \hat{f}_k(\x) + \varphi_k(\x)
$
where $\hat{f}_k(\x) = \frac{1}{2} \x^\top \A \x + \bv_k^\top \x$. We can then interpret~\eqref{eq:general-cost} as a \textit{perturbation of the quadratic $\hat{f}_k$}, where the perturbation has both a term whose gradient is bounded and another term whose gradient has bounded gain.
Consider now the scheme of Figure~\ref{fig:block-diagram} in which we replace $\nabla f_k(\cdot)$ with the perturbed gradient $\nabla \hat{f}_k(\cdot) + \nabla \varphi_k(\cdot)$. Clearly, if $\nabla \varphi_k(\x) = \0$ Proposition~\ref{pr:convergence-control-algorithm} applies\footnote{Indeed, by Assumption~\ref{as:extensions}(i) the cost $\hat{f}_k$ is strongly convex.} and we can control the tracking error to zero. We can then think of $\dv_k := \nabla \varphi_k(\x_k)$ as a disturbance, as depicted in Figure~\ref{fig:block-diagram-extensions}.
\begin{figure}[!ht]
\centering
\includegraphics[scale=0.8]{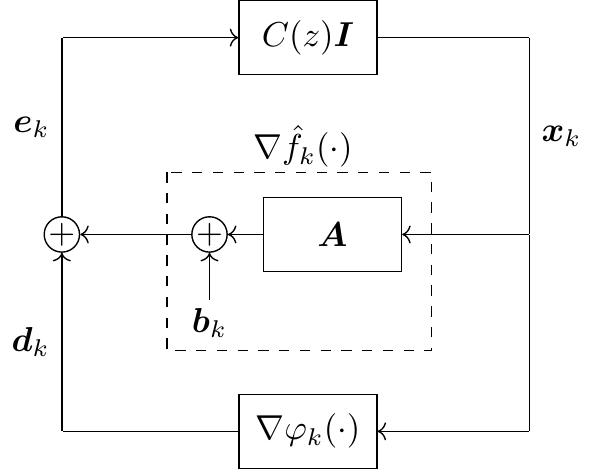}
\caption{The control scheme Figure~\ref{fig:block-diagram} applied to the general cost~\eqref{eq:general-cost}, highlighting the disturbance interpretation.}
\label{fig:block-diagram-extensions}
\end{figure}

\smallskip

The following section~\ref{ssubbsec:main-result} provides a convergence result for all non-quadratic costs that can be written as the ``perturbed'' quadratic~\eqref{eq:general-cost}, of which quadratic costs with time-varying Hessian are an example (see section~\ref{ssubsec:convergence-A-tv}). In section~\ref{ssubsec:inexact-model} we further analyze the convergence for quadratic problems when only an inexact knowledge of the internal model is available.

\subsubsection{Main convergence result}\label{ssubbsec:main-result}
The following result is a consequence of the \textit{small gain theorem}, which allows us to prove that the disturbance $\dv_k$ leads to a bounded (but, in general, non-zero) tracking error for the proposed algorithm~\eqref{eq:control-based-algorithm}. The proof is given in the appendix.

\begin{proposition}[Convergence for~\eqref{eq:generic-online-optimization}]\label{pr:generic-online-optimization}
Consider problem~\eqref{eq:generic-online-optimization} for which Assumption~\ref{as:extensions} holds. 
Assume that the controller $C(z) = \frac{C_\mathrm{N}(z)}{C_\mathrm{D}(z)}$ such that
\begin{enumerate}
	\item[(c1)] \textit{internal model}: $C_\mathrm{D}(z)$ includes all poles of $B_\mathrm{D}(z)$,
	
	\item[(c2)] \textit{stability}: $B_\mathrm{D}(z) - C_\mathrm{N}(z) \lambda$ is stable for all $\lambda \in [\lmin, \lmax]$,
	
	\item[(c3)] \textit{small gain}: $\norm{C(z) (\I - C(z) \A)^{-1}}_\infty < \frac{1}{\gamma}$.
\end{enumerate}
Then the output $\{ \x_k \}_{k \in \N}$ of the online algorithm~\eqref{eq:control-based-algorithm} is such that
\begin{align}
	&\limsup_{k \to \infty} \norm{\x_k - \x_k^*} \leq \nonumber \\ &\quad \norm{(\I - C(z) \A)^{-1}}_\infty \frac{\delta + \beta \gamma \norm{C(z) (\I - C(z) \A)^{-1}}_\infty}{1 - \gamma \norm{C(z) (\I - C(z) \A)^{-1}}_\infty}. \label{eq:asymptotic-error-general}
\end{align}
\smallskip
\end{proposition}

\smallskip

The result we derived depends on the $\infty$-norm of the two transfer matrices $(\I - C(z) \A)^{-1}$ and $C(z) (\I - C(z) \A)^{-1}$. The following lemma provides bounds to these norms that are useful in practice when numerically designing the controller.

\begin{lemma}\label{lem:bounds-infinity-norm}
We have the following bounds:
\begin{align*}
	\norm{(\I - C(z) \A)^{-1}}_\infty &\leq \max_{\lambda \in [\lmin, \lmax]} \norm{(1 - \lmax C(z))^{-1}}_\infty, \\
	\norm{C(z) (\I - C(z) \A)^{-1}}_\infty &\leq \max_{\lambda \in [\lmin, \lmax]} \norm{C(z) (1 - \lmax C(z))^{-1}}_\infty.
\end{align*}
\end{lemma}
\begin{proof}
Consider the eigendecomposition $\A = \V \Lm \V^\top$, then by definition we have
\begin{align*}
	\norm{(\I - C(z) \A)^{-1}}_\infty &= \norm{\V (\I - C(z) \Lm)^{-1} \V^\top}_\infty \\
	&\leq \norm{(\I - C(z) \Lm)^{-1}}_\infty \\
	&= \max_{\theta \in [0, 2\pi]} \sigma_{\max}\left( (\I - C(e^{i \theta}) \Lm)^{-1} \right) \\
	&\leq \max_{\theta \in [0, 2\pi]} \max_{\lambda \in [\lmin, \lmax]} |1 - C(e^{i \theta}) \lambda|^{-1}
\end{align*}
where we used Assumption~\ref{as:extensions}(i). The thesis follows by swapping the maxima (which we can do since they are defined on compacts). The second bound follows using the same arguments.
\end{proof}

\subsubsection{Time-varying $\A_k$}\label{ssubsec:convergence-A-tv}
We consider now the following quadratic, online optimization problem
\begin{equation}\label{eq:online-quadratic-A-tv}
	\x_k^* = \argmin_{\x \in \R^n} f_k(\x) := \frac{1}{2} \x^\top \A_k \x +  \bv_k^\top \x 
\end{equation}
in which, differently from~\eqref{eq:online-quadratic}, also the quadratic term is time-varying.
We prove the following convergence result in which the Hessian $\A_k$ is modeled as the sum of a static matrix and a time-varying perturbation.

\begin{corollary}[Time-varying $\A_k$]\label{cor:time-varying-A}
Consider problem~\eqref{eq:online-quadratic-A-tv} for which Assumption~\ref{as:generic-online-problem}, Assumption~\ref{as:extensions}~(i) and Assumption~\ref{as:extensions}~(ii) hold.
Further assume that we can write $\A_k = \A + \tilde{\A}_k$ with $\lmin \I \preceq \A \preceq \lmax \I$ and $\norm{\tilde{\A}_k} \leq \gamma$ for all $k$.
Assume that the controller $C(z) = \frac{C_\mathrm{N}(z)}{C_\mathrm{D}(z)}$ satisfies (c1),~(c2) of Proposition~\ref{pr:generic-online-optimization}.
Then the output $\{ \x_k \}_{k \in \N}$ of the online algorithm~\eqref{eq:control-based-algorithm} satisfies the bound~\eqref{eq:asymptotic-error-general} with $\delta = 0$.
\end{corollary}
\begin{proof}
We can rewrite the cost~\eqref{eq:online-quadratic-A-tv} as $f_k(\x) = \hat{f}_k(\x) + \varphi_k(\x)$ with $\hat{f}_k(\x) = \frac{1}{2} \x^\top \A \x +  \bv_k^\top  \x $ and $\varphi_k(\x) = \tilde\A_k  \x$. The cost thus conforms to~\eqref{eq:general-cost} with $\delta = 0$ since $\norm{\nabla \varphi_k(\x)} \leq \norm{\tilde\A_k} \norm{\x} \leq \gamma \norm{\x}$, and we can apply Proposition~\ref{pr:generic-online-optimization}.
\end{proof}

Corollary~\eqref{cor:time-varying-A} shows how a time-varying perturbation on the Hessian of the quadratic cost leads to a bounded asymptotic tracking error. Clearly if $\gamma = 0$ (the Hessian is static), then we recover the result of Proposition~\ref{pr:convergence-control-algorithm}.

\subsubsection{Inexact internal model}\label{ssubsec:inexact-model}
The convergence results of section~\ref{subsec:convergence-A-ti} were derived under the assumption that precise knowledge of the model for $\{\bv_k\}_{k\in\N}$ is available, or, more precisely, that we know exactly the denominator $B_\mathrm{D}(z)$ of its Z-transform.
In this section we are interested in evaluating the performance of the online algorithm \eqref{eq:control-based-algorithm} for the quadratic problem~\eqref{eq:online-quadratic} when it relies on an inexact knowledge of $B_\mathrm{D}(z)$, and specifically on the inexact model
\begin{equation}\label{eq:approxmod}
	\hat{B}_\mathrm{D}(z) = z^m + \sum_{i = 0}^{m-1}\hat b_i z^i.
\end{equation}
In this set-up we can derive the following result, whose proof follows a similar argument to Proposition~\ref{pr:generic-online-optimization} based on the small gain theorem.

\begin{proposition}[Inexact internal model]\label{cor:inexact-internal-model}
Consider problem~\eqref{eq:online-quadratic} for which Assumptions~\ref{as:generic-online-problem} and~\ref{as:extensions}(ii) hold.
Assume that the controller $C(z) = \frac{C_\mathrm{N}(z)}{C_\mathrm{D}(z)}$ with $C_\mathrm{D}(z) = \hat{B}_\mathrm{D}(z)$ is such that $\hat B_\mathrm{D}(z) - C_\mathrm{N}(z) \lambda$ is stable for all $\lambda \in [\lmin, \lmax]$. 
Then  the output $\{ \x_k \}_{k \in \N}$ of the online algorithm~\eqref{eq:control-based-algorithm} is such that
\begin{align*}
	&\limsup_{k \to \infty} \norm{\x_k - \x_k^*} \leq  \beta \norm{(\hat{B}_\mathrm{D}(z)\I - C_\mathrm{N}(z)  \A)^{-1}}_\infty \norm{\mathbold\delta}_1
\end{align*}
where
$
	{\mathbold\delta} = \begin{bmatrix} b_0 - \hat{b}_0 & \cdots & b_{m-1} - \hat{b}_{m-1}\end{bmatrix},
$

\end{proposition}

\begin{proof}
Let $\B(z)=\frac{\B_\mathrm{N}(z)}{B_\mathrm{D}(z)}$ be the Z-tranform of $\bv_k$.
Then, starting from~\eqref{eq:z-transform-e-eigendecomposition} it is straightforward to see that $\E(z)=\E'(z)+\E''(z)$
where 
\begin{align*}
	\E'(z) &=   \left(  \I + C(z) \A \right)^{-1} \frac{\B_N(z)}{\hat B_D(z)} \\
	\E''(z) &=   \left(  \I + C(z) \A \right)^{-1}\frac{{B}_\mathrm{D}(z) - \hat B_\mathrm{D}(z)}{\hat B_\mathrm{D}(z)}\B(z) 
\end{align*}
From Proposition \ref{pr:convergence-control-algorithm} we know that the inverse Z-transform $\e'_k$ of $\E'(z)$ is converging to zero. As far as $\E''(z)$ 
by applying (r1) in the appendix we can argue that its inverse Z-transform $\e''_k$ is such that
$$\norm{\e''_k}\leq \norm{ \left(  \I + C(z) \A \right)^{-1}\frac{{B}_\mathrm{D}(z) - \hat B_\mathrm{D}(z)}{\hat B_\mathrm{D}(z)}}_\infty \beta$$
Observe finally that
\begin{align*}
	&\norm{ \left( \I + C(z) \A \right)^{-1}\frac{{B}_\mathrm{D}(z) - \hat B_\mathrm{D}(z)}{\hat B_\mathrm{D}(z)}}_\infty \leq \\
	&\qquad \norm{(\hat{B}_\mathrm{D}(z)\I - C_\mathrm{N}(z)  \A)^{-1}}_\infty \ \norm{ {B}_\mathrm{D}(z) - \hat B_\mathrm{D}(z)}_\infty
\end{align*}
where we have
\begin{align*}
	\norm{ {B}_\mathrm{D}(z) - \hat B_\mathrm{D}(z)}_\infty&= \max_{\theta \in [0, 2\pi]}
	\left|\sum_{j=0}^{m-1} (b_j-\hat b_j ) e^{ij\theta}\right|\\
	&\le \sum_{j=0}^{m-1} |b_j-\hat b_j|=\norm{\mathbold\delta}_1,
\end{align*}
and where a bound for $\norm{(\hat{B}_\mathrm{D}(z)\I - C_\mathrm{N}(z)  \A)^{-1}}_\infty$ can be derived along the lines of Lemma~\ref{lem:bounds-infinity-norm}.
\end{proof}

We remark that if the internal model is exact then $\hat{B}_\mathrm{D}(z) = {B}_\mathrm{D}(z)$ and we recover the result of Proposition~\ref{pr:convergence-control-algorithm}.

\section{Simulations}\label{sec:simulations}

\subsection{Time-varying linear term}\label{subsec:numerical-tv-linear}
In this section we compare the performance of the different proposed algorithms when applied to problem~\eqref{eq:online-quadratic}, where only the linear term is time-varying\footnote{All the simulations were implemented using \texttt{tvopt} \cite{bastianello_tvopt_2021}.}.

We consider problem~\eqref{eq:online-quadratic} with $\x \in \R^n$, $n = 500$, and $\A = \V \Lm \V^\top$, where $\V$ is a randomly generated orthogonal matrix and $\Lm$ is diagonal with elements in $[1, 10]$. We employ four different models for the linear term $\{ \bv_k \}_{k \in \N}$: \emph{(i)} ramp $\bv_k = k \Ts \bar{\bv}$, $\bar{\bv} \in \R^n$; \emph{(ii)} sine $\bv_k = \sin\left( \omega k \Ts \right) \1$, $\omega = 1$; \emph{(iii)} sine plus ramp $\bv_k = \sin\left( \omega k \Ts \right) \1 + k \Ts \bar{\bv}$; \emph{(iv)} squared sine $\bv_k = \sin^2\left( \omega k \Ts \right) \1$.

In Figure~\ref{fig:tv_linear} we report the tracking error for the online gradient \cite{dallanese_optimization_2020}, the predicted online gradient\footnote{This algorithm is characterized by the update $\x_{k+1} = \x_k - \alpha \left(2 \nabla f_k(\x_k) - \nabla f_{k-1}(\x_k)\right)$, $k \in \N$.} \cite{bastianello_primal_2020}, and the control-based method~\eqref{eq:control-based-algorithm}.
\begin{figure}[!ht]
\centering
\includegraphics[width=0.475\textwidth]{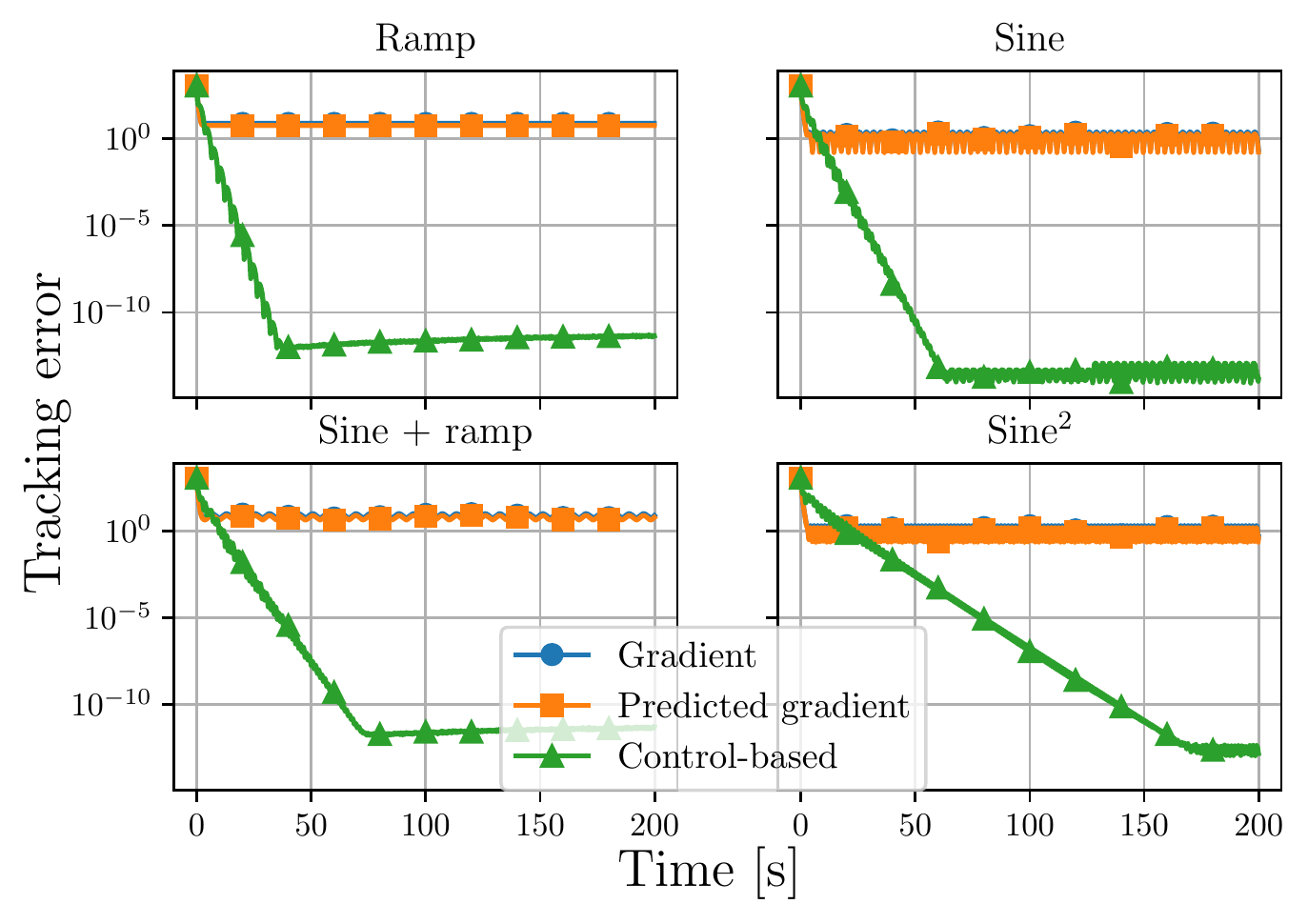}
\caption{Tracking error comparison for~\eqref{eq:online-quadratic} with different $\bv_k$.}
\label{fig:tv_linear}
\end{figure}
In accordance with the theoretical results of section~\ref{sec:control-based-algorithm}, the control-based method (with adequately chosen internal model) can achieve practically zero tracking error. On the other hand, both the gradient and predicted gradient only achieve a non-zero tracking error, with the error of the latter being smaller. This is further highlighted in Table~\ref{tab:ae-tv-linear}, which reports the \emph{asymptotic tracking error}\footnote{Computed as the maximum error over the last $4/5$ of the simulation.} for the different algorithms and problem models.

\begin{table}[!ht]
\centering
\caption{Asymptotic tracking error of the proposed algorithms for different models of $\bv_k$ in~\eqref{eq:online-quadratic}.}
\begin{tabular}{ccccc}
	Algorithm   & Ramp            & Sine   			& Sine + ramp 	  & Sine$^2$        \\
	\hline
	Grad.       & $7.61$ 	      & $2.44$			& $10.00$        & $2.03$			\\
	Pred. grad. & $5.81$         & $1.87$ 	    & $7.73$         & $1.55$			\\
	Control  	& \num{5.13e-12} & \num{1.21e-13} & \num{5.35e-12} & \num{1.93e-12} \\
    \hline
\end{tabular}
\label{tab:ae-tv-linear}
\end{table}

\smallskip

The previous numerical results were derived by applying the control-based~\eqref{eq:control-based-algorithm} equipped with an exact model of $\bv_k$ dynamics. We turn now to evaluating the performance of the proposed method when an inexact model is available, as discussed in section~\ref{subsec:convergence-A-ti}.
Consider the case~\emph{(ii)} of a sinusoidal linear term, with $\omega = 1$ being the unique parameter that defines the model of $\bv_k$, indeed $B_\mathrm{D}(z) = z^2 - 2 \cos(\omega \Ts) z - 1$. We run~\eqref{eq:control-based-algorithm} equipped with a (possibly inexact) model characterized by $\hat{b}_1 = 2 \cos(\Ts \hat{\omega})$, where $\hat{\omega}$ ranges in $[0.5, 1]$.

\begin{figure}[!ht]
\centering
\includegraphics[width=0.475\textwidth]{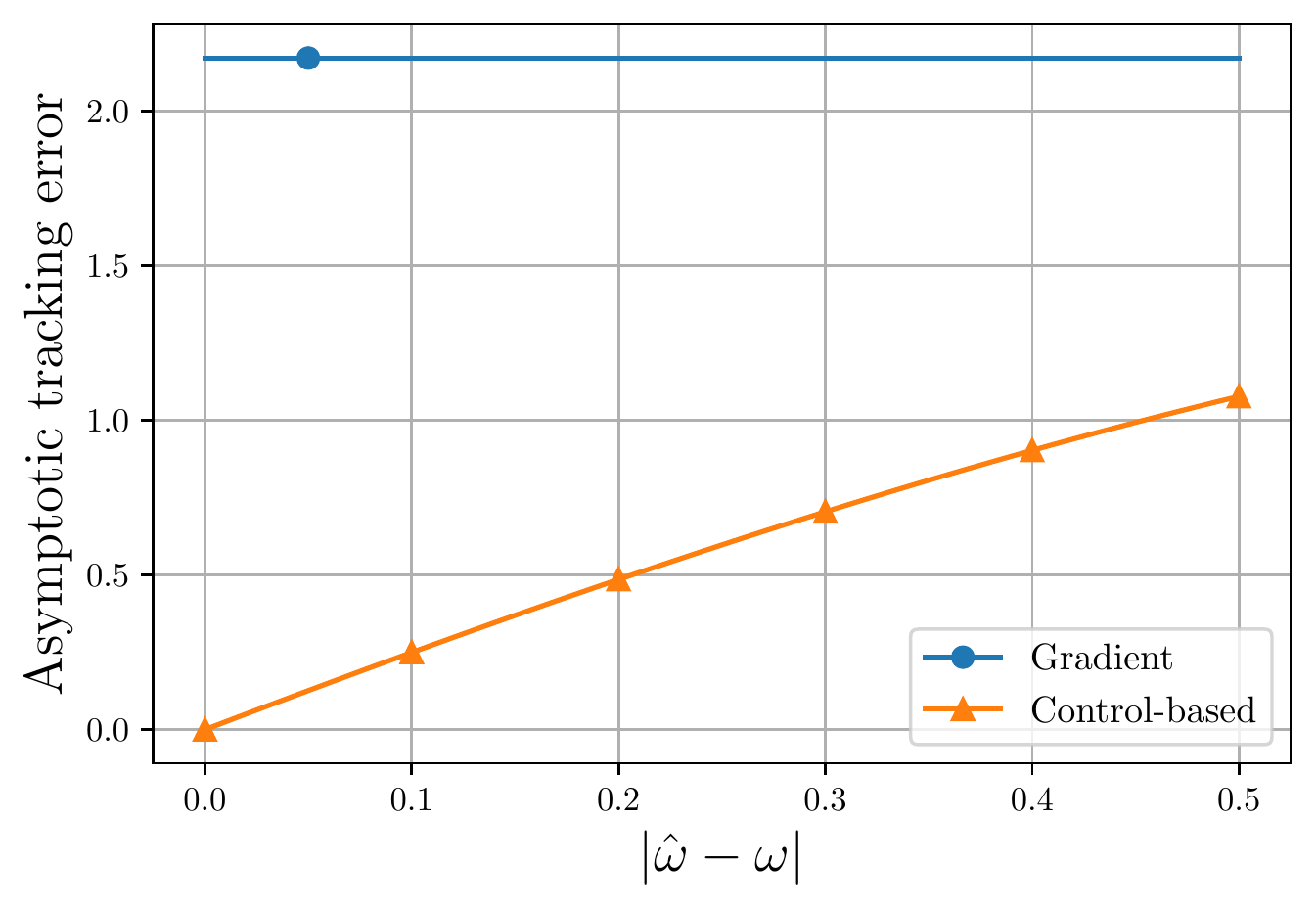}
\caption{Asymptotic error of~\eqref{eq:control-based-algorithm} when an approximate model of the sinusoidal $\bv_k$ (specifically, of its frequency) is employed.}
\label{fig:model_mismatch}
\end{figure}

Figure~\ref{fig:model_mismatch} depicts the resulting asymptotic tracking error of~\eqref{eq:control-based-algorithm} compared with the online gradient. As we can see, the performance of~\eqref{eq:control-based-algorithm} degrades as $|\hat{\omega} - \omega|$, in accordance with Proposition~\ref{pr:inexact-model}. However,~\eqref{eq:control-based-algorithm} still achieves better results than the online gradient, even when $\hat{\omega}$ is wrong by $50\%$ of the actual value of $\omega$.

\subsection{Time-varying quadratic term}\label{subsec:numerical-tv-quadratic}
We consider now problem~\eqref{eq:online-quadratic-A-tv} where $\A_k = \A + \tilde{\A}_k$, with $\A = \V \Lm \V^\top$ and $\tilde{\A}_k = \V \diag\{ \sin(\omega k \Ts) \dv \} \V^\top$ ($\dv \in \R^n$ being a vector with decreasing components), and $\Lm$ chosen such that $\A_k$ have eigenvalues in $[1, 10]$. For simplicity, the linear term is assumed constant, $\bv_k = \bar{\bv} \in \R^n$, $n = 500$.

As discussed in section~\ref{subsec:online-algorithm}, the proposed algorithm can be applied in this scenario, and in Figure~\ref{fig:tv_quadratic} we report its tracking error as compared to predicted and online gradient. Since to define algorithm~\eqref{eq:control-based-algorithm} we need to specify a model, we compared the three different options (cf. Example~\ref{ex:periodic})
\begin{equation}\label{eq:models-used}
	B_\mathrm{D}(z) = (z - 1) \prod_{\ell = 1}^L (z^2 - 2\cos(\ell \omega \Ts) z + 1), \quad L = 1, 2, 3.
\end{equation}

\begin{figure}[!ht]
\centering
\includegraphics[width=0.475\textwidth]{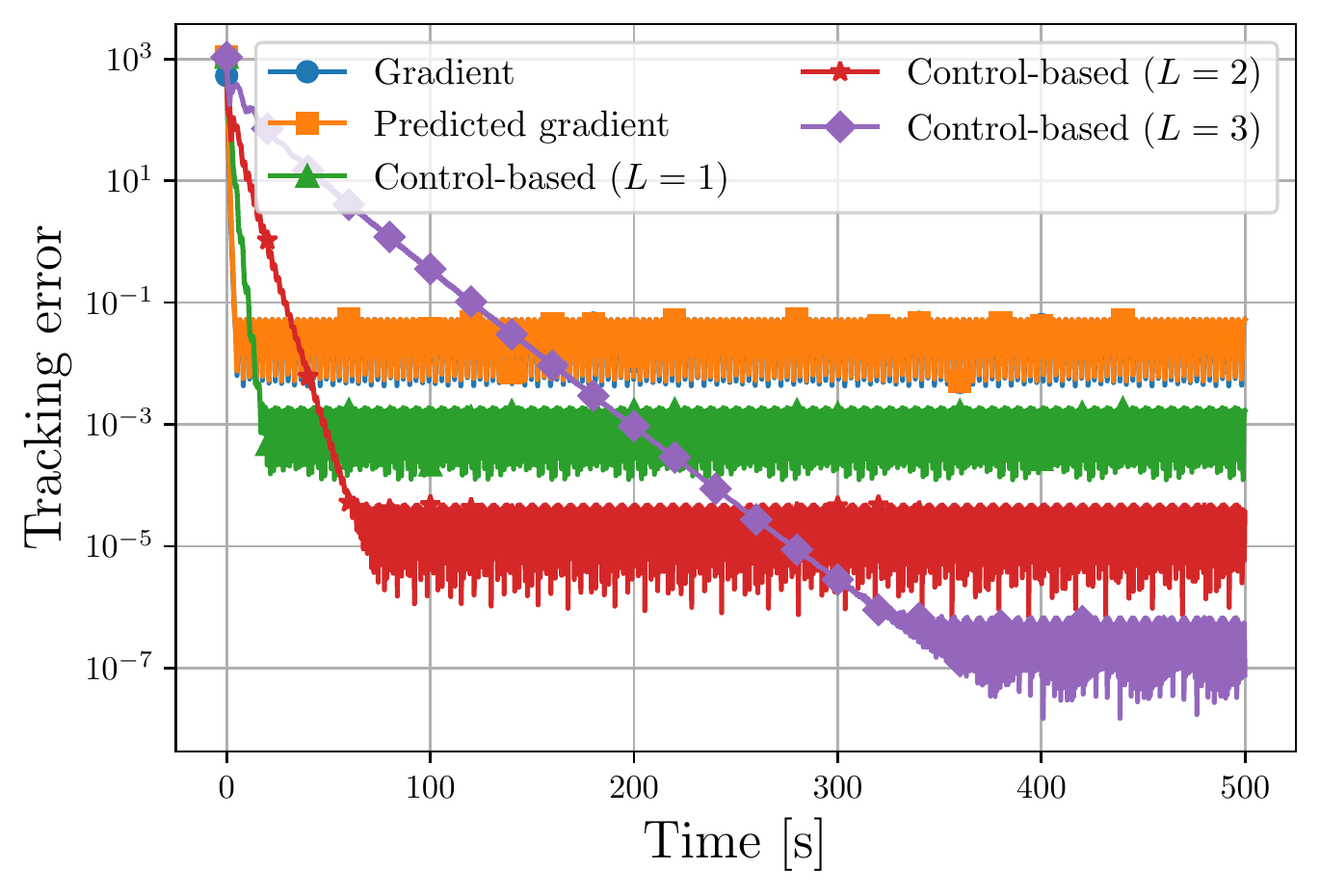}
\caption{Tracking error for the proposed algorithms applied to~\eqref{eq:online-quadratic-A-tv}, with the control-based algorithm~\eqref{eq:control-based-algorithm} employing three different approximate internal models.}
\label{fig:tv_quadratic}
\end{figure}

\begin{table}[!ht]
\caption{Comparison of asymptotic tracking errors}
\begin{subtable}{.5\linewidth}
\centering
\caption{Quadratic, $\A_k$ time-varying}
\begin{tabular}{cc}
	Algorithm   		& Asympt. err.    	\\
	\hline
	Grad.       		& \num{5.304e-2}  	\\
	Pred. grad. 		& \num{5.294e-2}    \\
	Control ($L = 1$)  	& \num{1.887e-3} 	\\
	Control ($L = 2$)  	& \num{4.793e-5} 	\\
	Control ($L = 3$)  	& \num{6.740e-7} 	\\
    \hline
\end{tabular}
\label{tab:ae-tv-quadratic}
\end{subtable}%
\begin{subtable}{.5\linewidth}
\centering
\caption{Non-quadratic}
\begin{tabular}{cc}
	Algorithm   		& Asympt. err.    	\\
	\hline
	Grad.       		& \num{2.823e-2}  	\\
	Pred. grad. 		& \num{1.901e-2}    \\
	Control ($L = 1$)  	& \num{1.323e-3} 	\\
	Control ($L = 2$)  	& \num{5.978e-5} 	\\
	Control ($L = 3$)  	& \num{1.662e-6} 	\\
    \hline
\end{tabular}
\label{tab:non-quadratic}
\end{subtable}
\end{table}

As we can see, the control-based method outperforms the other methods for all three choices of $L$, and we notice that larger models ($L$ larger) yield better performance, since they consider a higher number of multiples of the base frequency. Table~\ref{tab:ae-tv-quadratic} reports the exact asymptotic errors of the compared methods. It is important to notice that the control-based algorithm has a slower convergence rate than ``model-agnostic'' methods, which means that only after a longer transient does it reach a smaller tracking error. But considering that our goal is to improve the tracking error in the long run (cf. Example~\ref{ex:online-learning}), the trade-off of with a longer transient is justified.

\subsection{Non-quadratic}\label{subsec:numerical-non-quadratic}
In this section we discuss the application of the proposed algorithms to the non-quadratic problem~\eqref{eq:generic-online-optimization}. In particular, we consider the following cost function, adapted from \cite{simonetto_class_2016}:
\begin{equation}\label{eq:non-quadratic-cost}
	f_k(\x) = \frac{1}{2} \x^\top \A \x + \langle \bv, \x \rangle + \sin(\omega k \Ts) \log\left( 1 + \exp \langle \cv, \x \rangle \right)
\end{equation}
where $\omega = 1$, $\bv \in \R^n$, $n = 500$, is randomly generated, and $\cv \in \R^n$ is such that $\norm{\cv} = 1$. The matrix $\A$ is generated as in section~\ref{subsec:numerical-tv-linear}, and $f_k$ thus satisfies Assumption~\ref{as:extensions} with $\delta = \norm{\cv}^2/4$ and $\gamma = 0$.
In Figure~\ref{fig:non_quadratic} we report the tracking error of the proposed algorithm, with the control-based one using the three models~\eqref{eq:models-used}.
\begin{figure}[!ht]
\centering
\includegraphics[width=0.475\textwidth]{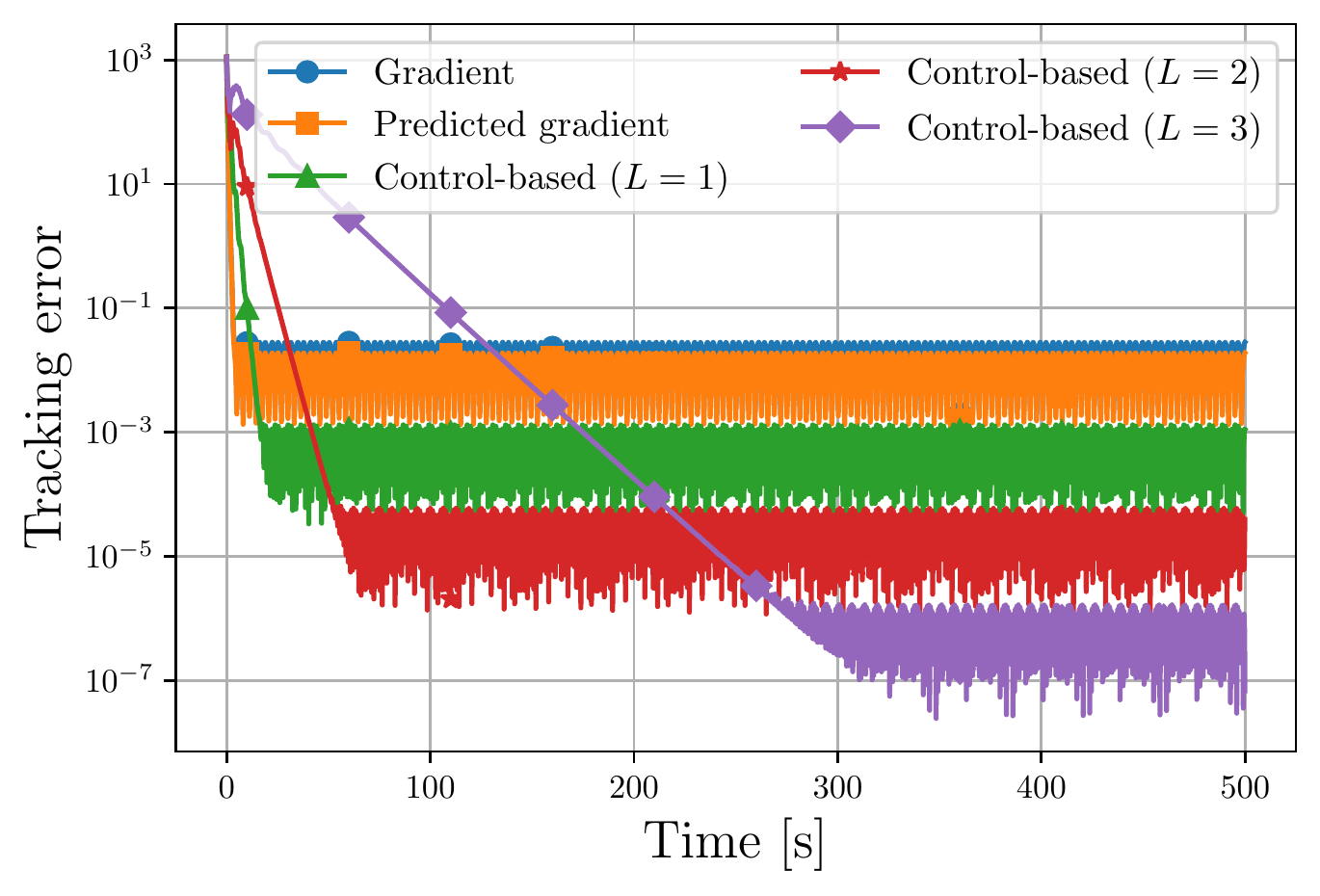}
\caption{Tracking error for the proposed algorithms applied to the non-quadratic~\eqref{eq:generic-online-optimization} with~\eqref{eq:non-quadratic-cost}.}
\label{fig:non_quadratic}
\end{figure}
As we can see, similarly to the case of section~\ref{subsec:numerical-tv-quadratic} above, the control-based method outperforms the other algorithms, to the cost of a slower convergence rate. Table~\ref{tab:non-quadratic} reports the exact asymptotic error achieved by the different algorithms.

\section{Conclusions}\label{sec:conclusions}
In this paper we proposed a model-based approach to the design of online optimization algorithms, with the goal of improving the tracking of the solution trajectory w.r.t. state-of-the-art methods. Focusing first on quadratic problems, we have proposed a novel online algorithm that achieves zero tracking error. Secondly, we have discussed the use of this algorithm for more general costs, and analyzed its convergence. The numerical results that we present validate our theoretical results and show the promise of our approach to online optimization algorithms design which outperforms state-of-the-art methods.
Future research directions that we will explore are the application of this paradigm to constrained problems (as stemming \emph{e.g.} in model predictive control), by designing primal-dual online algorithms that can handle time-varying equality and inequality constraints.

\appendix
\section{Proof of Proposition \ref{pr:generic-online-optimization}}
Before proceeding with the proof of Proposition \ref{pr:generic-online-optimization}, we briefly recall some relevant notions, see \cite{vidyasagar2011control} for a comprehensive introduction.

In the following, for any signal $\av_k \in \R^n$, $k\geq 0$, let $\norm{\av_k}_{\infty}=\sup_{k\geq 0} \norm{\av_k}$.
\begin{enumerate}
	\item[(r1)] Consider the system with input $\uv_k\in\R^n$, output $\y_k\in\R^n$ and stable transfer matrix $\T(z) \in \R^{n \times n}[z]$ so that the $\mathcal{Z}$-transforms $\U(z) = \mathcal{Z}[\uv_k]$ and $\Y(z) = \mathcal{Z}[\y_k]$ are in the relation $\Y(z) = \T(z) \U(z)$. 
Then we know that if $\norm{\uv_k}_\infty \leq \beta$, then $\norm{\y_k}_\infty \leq \norm{\T(z)}_\infty \beta$, where
$$\norm{(T(z)}_\infty :=  \max_{\theta \in [0, 2\pi]} \sigma_{\max}\left( T(e^{i \theta}) ) \right)$$

	\item[(r2)] Moreover, assume now that $\T(z)$ is in feedback with a (possibly non-linear) operator $\bphi : \R^n \to \R^n$, such that $\uv_k = \bphi(\y_k)$. If $\bphi$ is such that $\norm{\bphi(\x)} \leq \gamma \norm{\x}$ for $\gamma > 0$, then, according to the small gain theorem, if $\norm{\T(z)}_\infty \gamma < 1$, then the interconnection is asymptotically stable.
\end{enumerate}

\begin{proof}[of Proposition \ref{pr:generic-online-optimization}]

Let $\dv_k=\nabla \varphi_k(\x_k))$ and let $\D(z) = \mathcal{Z}[\dv_k]$. Following the same steps leading to~\eqref{eq:z-transform-e}, we can write
\begin{equation}\label{eq:sum-errors}
	\E(z) = (\I - C(z) \A)^{-1} (\B(z) + \D(z)) =: \E^b(z) + \E^d(z)
\end{equation}
where $\E^b(z)$ is given in \eqref{eq:z-transform-e} and $\E^d(z) = (\I - C(z) \A)^{-1} \D(z)$. The results of section~\ref{subsec:convergence-A-ti} imply that, by choosing a controller that satisfies (c1) and (c2), the first component $\e_k^b = \mathcal{Z}^{-1}[\E^b(z)]$ converges to zero. Hence we can focus on the second term $\e_k^d = \mathcal{Z}^{-1}[\E^d(z)]$.

By (r1) we have $\norm{\e_k^d}_\infty \leq \norm{(\I - C(z) \A)^{-1}}_\infty \norm{\dv_k}_\infty$ and, using Assumption~\ref{as:extensions}(iii),
\begin{equation}\label{eq:norm-d}
	\norm{\dv_k}_\infty \leq \norm{\nabla \varphi_k'(\x_k)} + \norm{\nabla \varphi_k''(\x_k)} \leq \delta + \gamma \norm{\x_k}_\infty.
\end{equation}
We need therefore to ensure that $\norm{\x_k}$ is bounded for all $k \in \N$ in order to guarantee that the disturbance is bounded as well.

By the fact that $\X(z) = C(z) \E(z)$, and using \eqref{eq:sum-errors} and (r1), we can write $\norm{\x_k}_\infty \leq \norm{C(z) (\I - C(z) \A)^{-1}}_\infty (\norm{\bv_k}_\infty + \norm{\dv_k}_\infty)$. Using now~\eqref{eq:norm-d} and the bound $\norm{\bv_k} \leq \beta$ (cf. Assumption~\ref{as:extensions}(ii)) we get
\begin{equation}\label{eq:norm-x}
	\norm{\x_k} \leq \norm{C(z) (\I - C(z) \A)^{-1}}_\infty \left( \beta + \delta + \gamma \norm{\x_k}_\infty \right).
\end{equation}
Clearly the first two terms are bounded, and we need to guarantee that $\norm{C(z) (\I - C(z) \A)^{-1}}_\infty \gamma \norm{\x_k}_\infty$ is as well. But this can be done by applying the small gain theorem recalled in (r2), therefore $\norm{\x_k}_\infty$ is bounded provided that $C(z)$ satisfies (c3).
Finally, using~\eqref{eq:norm-x} into~\eqref{eq:norm-d} and rearranging we get
$$
	\norm{\dv_k}_\infty \leq \frac{\delta + \beta \gamma \norm{C(z) (\I - C(z) \A)^{-1}}_\infty}{1 - \gamma \norm{C(z) (\I - C(z) \A)^{-1}}_\infty}.
$$
By $\norm{\e_k^d}_\infty \leq \norm{(\I - C(z) \A)^{-1}}_\infty \norm{\dv_k}_\infty$ the thesis follows.
\end{proof}

\smallskip

\bibliographystyle{IEEEtran}
\bibliography{references}

\end{document}